%% file: automorphisms.tex
\documentclass[twoside, a4paper,10pt]{article} 
 
\include{include} 

\begin{document}

\title{Outer automorphisms of free Burnside groups}
\author{R\'emi Coulon \\ \small{Max-Planck-Institut f\"ur Mathematik}\\\small{Vivatsgasse 7,
53111 Bonn, Germany} \\ \small{\texttt{coulon@mpim-bonn.mpg.de}}}

\maketitle

\abstract{
	In this paper, we study some properties of the outer automorphism group of free Burnside groups of large odd exponent.
	In particular, we prove that it contains free and free abelian subgroups.
}

\tableofcontents 

\input{0_introduction}

\input{1_outB}

\input{2_small_cancellation} 
\input{3_main_theorem} 

\def\cprime{$'$}  

\nocite{BesFei96} 
\nocite{BesFeiHan97a} 
\bibliography{bibliography} 
\bibliographystyle{alpha} 

\end{document}

%% file: include.tex
\usepackage{ucs}
\usepackage[utf8x]{inputenc}
\usepackage[T1]{fontenc}
\usepackage[english]{babel}
\usepackage[dvips]{graphicx} 
\usepackage{subfigure}
\usepackage{amsmath, amsthm, amsfonts, amssymb}
\usepackage{latexsym}
\usepackage{color}
\usepackage[all]{xy}
\usepackage{stmaryrd}
\usepackage[dvips]{epsfig}
\usepackage{endnotes}
\usepackage{tikz} 
\usetikzlibrary{matrix,arrows,calc,backgrounds,decorations.markings,decorations.text}
\usepackage{wrapfig}
\usepackage{fancyhdr,lastpage}
\usepackage{ccaption}
\usepackage{enumitem}
\usepackage{xspace}
\usepackage{ifthen}
\usepackage{xkeyval}
\usepackage{remreset}
\usepackage{xcolor}

\makeatletter

\let\orig@enumerate =\enumerate
\renewenvironment{enumerate}
{\orig@enumerate [label=(\roman{*}), ref=(\roman{*})]}
{\endlist}

\definecolor{grispuce}{gray}{0.7}

\newcommand{\ops}{$\omega$-presque surement ($\omega$-ps)\xspace \renewcommand{\ops}{$\omega$-ps\xspace}}
\newcommand{\oeb}{$\omega$-essentiellement borné  ($\omega$-eb)\xspace \renewcommand{\oeb}{$\omega$-eb\xspace}}

\newtheorem{defi}{Definition}[section]
\newtheorem{lemm}[defi]{Lemma}
\newtheorem{prop}[defi]{Proposition}
\newtheorem{theo}[defi]{Theorem}

\@addtoreset{defi}{section}
\renewcommand{\thesection}{%
	\arabic{section}%
}
\renewcommand{\thesubsection}{%
	\ifnum \value{section}>0
		\thesection.%
	\else%
	\fi%
	\arabic{subsection}%
}
\renewcommand{\thedefi}{%
	\ifnum \value{section}>0
		\thesection.%
	\else%
	\fi%
	\arabic{defi}%
}

\newcommand{\rem}{\paragraph{Remark :}}

\newcommand{\ex}{\paragraph{Example :}}

\newcommand{\N}{\mathbf{N}}
\newcommand{\Z}{\mathbf{Z}}
\newcommand{\Q}{\mathbf{Q}}
\newcommand{\R}{\mathbf{R}}
\newcommand{\C}{\mathbf{C}}
\renewcommand{\H}{\mathbf{H}}
\newcommand{\intval}[2]{\left[#1\, , #2\right]}
\newcommand{\intvald}[2]{\left\{ #1,\dots , #2\right\} }

\define@choicekey*{planhyp}{corps}[\val \nr]{r,c,h,}[]{
	\ifcase\nr 
		\def\ph@corps{\R}
	\or
		\def\ph@corps{\C}
	\or
		\def\ph@corps{\H}
	\or
		\def\ph@corps{}
	\fi
}
\presetkeys{planhyp}{corps}{}
\newcommand{\HP}[2][]{
	\begingroup
		\setkeys{planhyp}{#1}
		\ifthenelse{\equal{\ph@corps}{}}
			{\mathbf{H}_{#2}}
			{\mathbf{H}_{#2}\left(\ph@corps \right)}
	\endgroup
}

\newcommand*{\dist}[3][]{
	\ifthenelse{\equal{#1}{}}
		{\left| #2- #3 \right|}
		{
			\ifthenelse{\equal{#1}{SC}}
			{\left\| #2- #3 \right\|}
			{\left| #2- #3 \right|_{#1}}
		}
}
\newcommand*{\distV}[1][]{
	\ifthenelse{\equal{#1}{}}
		{\left| \ . \ \right|}
		{
			\ifthenelse{\equal{#1}{SC}}
			{\left\| \ . \ \right\|}
			{\left| \ . \ \right|_{#1}}
		}
}
\define@boolkey{longueurtrans}{stable}[true]{}
\define@key{longueurtrans}{espace}[]{\def \lt@espace{#1}}
\presetkeys{longueurtrans}{stable=false,espace}{}
\newcommand*{\len}[2][]{
	\begingroup
		\setkeys{longueurtrans}{#1}
		\ifKV@longueurtrans@stable {
			\ifthenelse{\equal{\lt@espace}{}}
				{\left[ {#2}\right]^{\infty}}
				{\left[ {#2}\right]^{\infty}_{\lt@espace}}
		}
		\else {
			\ifthenelse{\equal{\lt@espace}{}}
				{\left[ {#2}\right]}
				{\left[ {#2}\right]_{\lt@espace}}
		}
		\fi
	\endgroup
}

\newcommand{\geo}[2]{\left[ #1, #2 \right]}
\newcommand{\rinj}[2]{r_{\textit{inj}}\left(#1,#2\right)}
\newcommand{\stab}[1]{\operatorname{Stab}\left( #1\right)}
\newcommand{\rips}[3][]{
	\ifthenelse{\equal{#1}{}}
		{P_{#2}\left( #3\right)}
		{P_{#2}^{(#1)}\left( #3\right)}
}
\newcommand{\diam}{\operatorname{diam}}

\newcommand{\sdp}[3][]{
	\ifthenelse{\equal{#1}{}}
		{#2 \rtimes #3}
		{#2 \rtimes _{#1} #3}
}
\newcommand{\aut}[1]{\operatorname{Aut} \left( #1\right)}
\newcommand{\out}[1]{\operatorname{Out} \left( #1\right)}
\newcommand{\free}[1]{\mathbf F_{#1}}
\newcommand{\F}{\mathbf F}
\newcommand{\burn}[2]{\mathbf B _{#1}(#2)}
\newcommand{\dlim}{\displaystyle{\lim_{\longrightarrow}}\ }

\renewcommand{\sinh}{\operatorname{sh}}
\renewcommand{\cosh}{\operatorname{ch}}

\newcommand{\id}{\operatorname{id}}
\newcommand{\rinjRF}[1]{r_{\textit{inj}}\left(#1\right)}

\newcommand{\fantomB}{\vphantom{\Big\vert}\!}
\renewcommand{\epsilon}{\varepsilon}
\renewcommand{\phi}{\varphi}
\renewcommand{\leq}{\leqslant}
\renewcommand{\geq}{\geqslant}

\makeatother

%% file: 0_introduction.tex
\section*{Introduction}

\paragraph{}
The free Burnside group of rank $r$ and exponent $n$, denoted by $\burn rn$, is the quotient of the free group $\free r$ by the subgroup $\free r^n$ generated by the $n$-th powers of all its elements.
In 1902, W.~Burnside asked whether $\burn rn$ has to be finite of not (see \cite{Bur02}).
For a long time, one only knows that the answer was positive for some small exponents (for $n=2$ see \cite{Bur02}, $n=3$ \cite{Bur02} and \cite{LevWae33}, $n=4$ \cite{San40} and $n=6$ \cite{Hal57}).
In 1968, P.S.~Novikov and S.I.~Adian achieved a breakthrough (see \cite{NovAdj68}, \cite{NovAdj68a} and \cite{NovAdj68b}).
Using the small cancellation theory, developed by V.A.~Tartakovski\u\i\  \cite{Tar49} and M.~Greendlinger \cite{Gre60}, \cite{Gre60a}, \cite{Gre61}, they proved that for large odd exponents, $\burn rn$ is infinite.
Thanks to a diagrammatic formulation of small cancellation, A.Y.~Ol'shanski\u\i\ simplified in 1982 the proof of P.S.~Novikov and S.I.~Adian \cite{Olc82}.
Recently, T.~Delzant and M.~Gromov, gave a more geometrical proof of the same theorem \cite{DelGro08}.
These works not only provide examples of infinite Burnside groups, they also study many of their properties (solution for the word-problem, description of finite subgroups,...).
Other information about the history of the Burnside problems can be found in \cite{GriLys02}.

\paragraph{}
The next step to understand Burnside groups is to study their automorphisms.
In this paper, we are interested in the following questions.
What kind of outer automorphisms of $\burn rn$ have infinite order?
Does $\out{\burn rn}$ contain relevant subgroups like free groups or free abelian groups?
To that end, we focus on the canonical map $\out{\free r} \rightarrow \out{\burn rn}$.

\paragraph{} Using the work of P.S.~Novikov and S.I.~Adian, E.A.~Cherepanov proved that the automorphism $\phi$ of $\free 2 = \mathbf F(a,b)$, defined by $\phi(a) =b$ and $\phi (b) = ab$, induces an infinite order outer automorphism of $\burn rn$ (see \cite{Che05} and Proposition \ref{theo:infinite order automorphism novikov}).
Our first theorem provides a large class of automorphisms of free group having the same property.

\begin{theo}[see Th. {\ref{theo:infinite order automorphism sc}}]
\label{intro:theo:infinite order automorphism}
	Let $\phi$ be an automorphism of $\free r$.
	Assume that $\phi$ is hyperbolic, i.e. the semi-direct product $\sdp[\phi]{\free r}{\Z}$ defined by $\phi$ is a hyperbolic group.
	There exists an integer $n_0$ such that for all odd integer $n$ larger than $n_0$, $\phi$ induces an infinite order outer automorphism of $\burn rn$.
\end{theo}

All proofs dealing with free Burnside groups have the same weakness: they involve a presentation of $\burn rn$ which is not stable under automorphisms.
Our work try to regain a little symmetry: we build a sequence of groups $(H_k)$ such that for all $k$, $\phi$ induces an automorphism of $H_k$ and $\dlim H_k = \burn rn$.
To that end, we start with $H_0 = \free r$ and, at each step, we construct $H_{k+1}$ as a small cancellation quotient of $H_k$.
Some difficulties appear during this process.
Assume that $\rho$ is one of the relations defining the first quotient $\free r = H_0 \twoheadrightarrow H_1$. 
Since we want $\phi$ to induce an automorphism of $H_1$, the elements $\phi^m(\rho)$, $m \in \N$, have to belong to the set of relations.
However the small cancellation theory only deals with relations having more or less the same length.
In our case, the relations $\phi^m(\rho)$ may have very different lengths.
To avoid this problem, we encode the information concerning the automorphism in a larger group: $\sdp[\phi]{\free r}{\Z}$.
Thus the elements $\phi^m(\rho)$ become conjugates of $\rho$ and do not need to be added to the set of relations.
We shall now use the fact that the group $\sdp {\free r}{\Z}$ is hyperbolic.
In 1991, A.Y.~Ol'shanski\u\i\  provided indeed a generalisation of the Novikov-Adian theorem (see \cite{Olc91}).
Given a torsion-free, hyperbolic group $G$, he proved that for large odd exponent $n$ the quotient $G/G^n$ is infinite.
This result was recovered by T. Delzant and M. Gromov in \cite{DelGro08}.
We would like to apply the same techniques to $G=\sdp {\free r}{\Z}$. 
However we must take care not to kill all $n$-th powers of $G$.
Indeed, if we did so the automorphism obtained at the end of the construction would have finite order dividing $n$.
That is why we propose an extension of the Delzant-Gromov construction where the relations are chosen in a normal subgroup of $\sdp {\free r}{\Z}$.
This construction works in a more general situation.
It leads to our main theorem

\begin{theo}[Main Theorem]
\label{intro:theo:main theorem}
	Let $1 \rightarrow H \rightarrow G \rightarrow F \rightarrow 1$ be a short exact sequence of groups. 
	Assume that $H$ is finitely generated, $G$ is hyperbolic, torsion-free and $F$ is torsion-free. 
	There exists an integer $n_0$ such that for all odd integers $n$ larger than $n_0$, the canonical map $F \rightarrow \out H$ induces an injective homomorphism $F \hookrightarrow \out{H/H^n}$.
\end{theo}

Theorem \ref{intro:theo:infinite order automorphism} is then an application of the main theorem to the short exact sequence $1 \rightarrow \free r \rightarrow \sdp {\free r}{\Z} \rightarrow \Z \rightarrow 1$.
The work of M.~Bestvina, M.~Feighn and M.~Handel, provides examples of hyperbolic extensions of free groups by free groups.
Using this result we obtain our second theorem.

\begin{theo}[see Th. {\ref{theo:free subgroups}}]
\label{intro:theo:free subgroup}
	Let $r \geq 3$.
	There exists an integer $n_0$ such that for all odd integers $n$ larger than $n_0$, the group $\out{\burn rn}$ contains a subgroup which is isomorphic to $\free 2$.
\end{theo} 

The strategy to embed abelian subgroups in $\out{\burn rn}$ is a little different.
We do not apply the main theorem to an appropriate hyperbolic extension of the free group.
We construct a family of automorphisms of $\free r$ which already commute in $\aut {\free r}$ and check ``by hand'' that they do not satisfy any other relation in $\out{\burn rn}$.
Thus, we obtain the following result.

\begin{theo}[see Th. {\ref{theo:abelian free subgroups}}]
\label{intro:theo:abelian subgroup}
	Let $r \geq 1$.
	There exists an integer $n_0$ such that for all odd integers $n$ larger than $n_0$, the groups $\out{\burn {2r}n}$ and $\out{\burn {2r+1}n}$ contains a subgroup which is isomorphic to $\Z^r$.
\end{theo}

\paragraph{} Hyperbolic automorphisms induce infinite order automorphisms of free Burnside groups of large exponent. 
But they are not the only one. 
For instance, the automorphism $\phi$ studied by E.A. Cherepanov, characterized by $\phi(a)=b$ and  $\phi(b)=ab$, is not hyperbolic. 
Indeed, $\phi^2$ fixes the commutator $\left[a^{-1},b^{-1}\right]$.
The semi-direct product $\sdp[\phi]{\free r}{\Z}$ contains therefore a subgroup which is isomorphic to $\Z^2$.
We wonder if there exists a criterion  to decide whether an automorphism of $\free r$ induces a infinite order outer automorphism of $\burn rn$ for some large exponent or not.
In particular, is there a link between this property and the growth of the automorphism?
Section~\ref{section:polynomially growing automorphisms} gives a partial answer.
We prove that a polynomially growing automorphism always induces a finite order automorphism of $\burn rn$.

\paragraph{Outline of the paper.}  In Section \ref{section:automorphisms} we explain the consequences of the main theorem. 
In particular, we provide examples of infinite order automorphisms of $\burn rn$.
We also construct free and free abelian subgroups of $\out{\burn rn}$.
Section \ref{section:small cancellation theory} deals with the proof of the main theorem.
At first, we recall the geometrical point of view on the small cancellation theory developed by T.~Delzant and M.~Gromov.
We also improve some results of \cite{DelGro08} which are necessary to control the small cancellation parameters in our situation.
Then, we prove an induction lemma (Lemma \ref{induction lemma}) which is the fundamental step of the induction process used in Section \ref{sec:proof main theorem} to prove the main theorem.

\paragraph{Acknowledgement.}  
I am grateful to Thomas Delzant for his invaluable help and advice during this work.
I would like to thank Gilbert Levitt for related discussions, in particular concerning the growth of automorphism.
Many thanks also go to  Étienne Ghys who points out many questions to me, like the embedding of free abelian subgroups.

%% file: 1_outB.tex
\section{Automorphisms of Burnside groups}
\label{section:automorphisms}

\rem In this paper, we are interested in outer automorphisms of free Burnside groups.
One question we look at is the following:
given an automorphism of the free group $\free r$, does it induce an infinite order automorphism of $\burn rn$ ?
Note that any element of the Burnside group has finite order.
In particular any inner automorphism of $\burn rn$ has finite order.
It follows that an element of $\aut{\burn rn}$ has finite order if and only if its image in $\out{\burn rn}$ has finite order.

\subsection{Examples of infinite order automorphisms}

Using the work of P.S.~Novikov and S.I.~Adian (\cite{NovAdj68}, \cite{NovAdj68a}, \cite{NovAdj68b}), we can produce a first example of an infinite order outer automorphism of Burnside groups.
This example was already studied by E.A.~Cherepanov in \cite{Che05}.

\begin{prop}[see {\cite[Th. 1]{Che05}}]
\label{theo:infinite order automorphism novikov}
Let $\left\{a,b\right\}$ be a generating set of the free group $\free 2$.
Let $\phi$ be the automorphism of $\free 2$ defined by $\phi(a) = ab$ and $\phi(b) =a$. 
There exists an integer $n_0$, such that for all odd integers $n$ larger than $n_0$, $\phi$ induces an infinite order automorphism of $\burn 2n$.
\end{prop} 

\begin{proof}
	We consider the sequence of iterated images of $a$ by $\phi$.
	\begin{displaymath}
		\begin{array}{lclclcl}
		\phi^0(a) & = & a			&\quad 	& \phi^4(a) & = & abaababa\\
		\phi^1(a) & = & ab			&		& \phi^5(a) & = & abaababaabaab\\
		\phi^2(a) & = & aba			&		& \phi^6(a) & = & abaababaabaababaababa\\
		\phi^3(a) & = & abaab		&		& \dots & &\\
		\end{array}
	\end{displaymath}
	This sequence converges to a right infinite positive word 
	\begin{displaymath}
		\phi^\infty(a) =  abaababaabaababaababa\dots
	\end{displaymath}
	which has the following property. 
	For every word $u$ in $\{a,b\}$, $u^4$ is not a subword of $\phi^\infty(a)$ (see \cite{Mos92}).
	Let $n$ be an odd integer larger than 10 000. 
	In order to prove that the free Burnside group of large exponent is infinite, P.S.~Novikov and S.I.~Adian use the following fact : if $m$ is a reduced word in $\{a,b\}$ which does not contain a subword equal to a fourth power, then $m$ defines a non-trivial element of $\burn 2n$ (see \cite[IV. 2.16.]{Adi79} or \cite[Statement 1]{AdiLys92}).
	In particular $\left(\phi^p(a)\right)$ induces a sequence of pairwise distinct elements of $\burn 2n$.
	It follows that $\phi$ induces an infinite order automorphism of $\burn 2n$.
\end{proof}

We are now interested in a large class of automorphisms of the free group: the hyperbolic automorphisms.
We prove that they all induce infinite order automorphisms of the free Burnside group.

\begin{defi}
	Let $G$ be a hyperbolic group.
	An automorphism $\phi$ of $G$ is hyperbolic if the semi-direct product $\sdp[\phi] G\Z$ defined by $\phi$ is hyperbolic. 
\end{defi}

\ex Let $\Sigma$ be the fundamental group of a compact surface $S$ of genus larger than 2.
Thanks to Thurston's hyperbolisation Theorem, any pseudo-Anosov homeomorphism of $S$ induces a hyperbolic automorphism of $\Sigma$ (see \cite{Ota96}).

\paragraph{} There exist many characterizations of hyperbolic automorphisms. 
Assume that $G$ is endowed with the word metric $\distV$ relative to a generating set, M.~Bestvina and M.~Feighn proved in \cite{BesFei92} that an automorphism $\phi$ of $G$ is hyperbolic if and only if there exist $\lambda > 1$ and $m \in \N$ such that for all $g \in G$,
\begin{displaymath}
	\lambda \left|g\right| \leq \max \left\{ \left|\vphantom{\phi^{-m}(g)} \phi^m(g)\right|, \left|\phi^{-m}(g) \right|\right\}.
\end{displaymath}
On the other hand, an automorphism of the free group is hyperbolic, if and only if it has no non-trivial periodic conjugacy classes (see \cite{BesFeiHan97} and \cite{Bri00}).
Note that the automorphism $\phi$ studied in proposition \ref{theo:infinite order automorphism novikov} is not hyperbolic:
$\phi^2$ fixes the commutator $\left[ a^{-1}; b^{-1}\right]$.
More generally, $\aut {\free 2}$ does not contain hyperbolic elements.
Any automorphism $\phi$ of $\free 2$ is indeed induced by a homeomorphism of the punctured torus.
Therefore $\phi$ has to fix the conjugacy class of $\free 2$ corresponding to the boundary of the torus.

\begin{theo}
\label{theo:infinite order automorphism sc}
	Let $r \geq 3$. Let $\phi$ be a hyperbolic automorphism of $\free r$.
	There exists an integer $n_0$, such that for all odd integers $n$ larger than $n_0$, $\phi$ induces an infinite order outer automorphism of $\burn rn$.
\end{theo}

\begin{proof}
	By definition, the group $\sdp[\phi] {\free r} \Z$ is hyperbolic.
	It follows that the short exact sequence $1 \rightarrow \free r \rightarrow \sdp[\phi] {\free r} \Z \rightarrow \Z \rightarrow 1$ satisfies the assumptions of the main theorem (see Theorem \ref{intro:theo:main theorem}).
	Thus there exists an integer $n_0$ such that for all odd integers $n$ larger than $n_0$, the map $\Z \rightarrow \out {\free r}$ induces an injective homomorphism $\Z \hookrightarrow \out{\burn rn}$.
	However, the morphism $\Z \rightarrow \out{\free r}$ is by construction the one that associates to an integer $m$ the outer automorphism induced by $\phi^m$.
	Consequently, $\phi$ induces an infinite order outer automorphism of $\burn rn$.
\end{proof}

\subsection{Polynomially growing automorphisms of free groups}
\label{section:polynomially growing automorphisms}
	\paragraph{} We provide now examples of infinite order automorphisms of  $\free r$ which induce finite order automorphisms of $\burn rn$.
	If $x$ is a conjugacy class of $\free r$, we denote by $[x]$ the length of any cyclically reduced word representing $x$.
	Given an outer automorphism $\Phi$ of $\free r$ we look at the action of $\Phi$ on the conjugacy classes of $\free r$.
	
	\begin{defi}
		The automorphism $\Phi$ grows polynomially if for every conjugacy class $x$ of $\free r$, the sequence $\left( \left[\Phi^p (x)\right] \right)$ grows polynomially.
	\end{defi}
	
	\begin{prop}[see {\cite{Lev08}}]
	\label{theo:decomposition polynomial automorphism}
		Let $\Phi$ be a polynomially growing outer automorphism of $\free r$.
		Up to replace $\Phi$ by a power of $\Phi$, one of the following assertion is true.
		\begin{enumerate}
			\item There exist $\phi \in \aut{\free r}$ representing $\Phi$ and a non-trivial free decomposition $F_1 * F_2$ of $\free r$ which is invariant under $\phi$.
			\item There exist $\phi \in \aut{\free r}$ representing $\Phi$, a non-trivial free decomposition $F_1*\left< t \right>$ of $\free r$ and an element $f$ of $F_1$ such that $F_1$ is invariant under $\phi$ and $\phi(t) = tf$.
		\end{enumerate}
	\end{prop}
	
	\begin{theo}
		Let $r \geq 1$.
		Let $\Phi$ be a polynomially growing outer automorphism of $\free r$.
		For all positive integers $n$, $\Phi$ induces a finite order outer automorphism of $\burn rn$.
	\end{theo}
	
	\begin{proof}
		We prove this result by induction on the rank $r$ of the free group.
		The outer automorphism group of $\Z$ is trivial. 
		Hence the theorem is true for rank one.
		Let $r \geq 1$.
		We assume now that the theorem is true for any rank that is smaller or equal to $r$.
		Let $\Phi$ be a polynomially growing outer automorphism of $\free {r+1}$ and $n$ a positive integer.
		Due to Proposition \ref{theo:decomposition polynomial automorphism}, we distinguish two cases.
		
		\paragraph{First case.} There exist an automorphism $\phi \in \aut{\free {r+1}}$ representing a power of $\Phi$  and a non-trivial free decomposition $F_1*F_2$ of $\free {r+1}$ invariant under $\phi$.
		We denote by $\phi_i$ the restriction of $\phi$ to $F_i$.
		By induction, there exists an integer $p_i$ such that $\phi_i^{p_i}$ induces the identity of $F_i /F_i^n$.
		It follows that $\phi^{p_1p_2}$ is trivial in $\aut{\burn {r+1}n}$.
		Therefore $\Phi$ induces a finite order outer automorphism of $\burn {r+1}n$.
		
		\paragraph{Second case.} There exist an automorphism $\phi \in \aut{\free {r+1}}$ representing a power of $\Phi$, a free decomposition $F_1 * \left<t \right>$ of $\free {r+1}$ and an element $f$ of $\F_1$ such that $F_1$ is invariant under $\phi$ and $\phi(t) = tf$.
		We denote by $\phi_1$ the restriction of $\phi$ to $F_1$. 
		By induction, there exists an integer $p_1$ such that $\phi_1^{p_1}$ induces the identity of $F_1/F_1^n$.
		On the other hand, for all integers $q$, $\phi^q(t)$ is equal to 
		\begin{math}
			tf \phi_1(f) \phi_1^2(f) \dots \phi_1^{q-1}(f)
		\end{math}.
		It follows that the below equality holds in $\burn {r+1}n$.
		\begin{displaymath}
			\phi^{np_1}(t) = t \left[ f \phi_1(f)\phi_1^2(f) \dots \phi_1^{p_1-1}(f)\right]^n = t
		\end{displaymath}
		Hence $\phi^{np_1}$ is trivial in $\aut{\burn {r+1}n}$.
		Therefore $\Phi$ induces a finite order outer automorphism of $\burn {r+1}n$.
	\end{proof}
	
\subsection{Subgroups of $\out{\burn rn}$}

We are now interested in relevant subgroups that can be embedded in $\out{\burn rn}$.
We start with free subgroups.
The following result is due to M.~Bestvina, M.~Feighn and M.~Handel

\begin{theo}[see {\cite[Th. 5.2]{BesFeiHan97}}]
\label{theo:hyperbolic extension of free group by free group}
	Let $r\geq 3$.
	Let $\phi_1$ and $\phi_2$ two automorphisms of $\free r$.
	We assume that the outer automorphisms induced by $\phi_1$ and $\phi_2$ are irreducible, do not have common powers and neither have a nontrivial periodic conjugacy class.
	There exists an integer $m$ such that $\phi_1^m$ and $\phi_2^m$ generate a free group.
	Moreover, the semi-direct product $\sdp {\free r}{\free 2}$ defined by $\phi_1^m$ and $\phi_2^m$ is hyperbolic.
\end{theo}	

\begin{theo}
\label{theo:free subgroups}
	Let $r\geq 3$.
	There exists  an integer $n_0$, such that for all odd integers $n$ larger than $n_0$, $\out{\burn rn}$ contains a subgroup which is  isomorphic to $\free 2$.
\end{theo}

\begin{proof}
	Theorem \ref{theo:hyperbolic extension of free group by free group} provides a hyperbolic extension of $\free r$ by $\free 2$.
	In other words, $1 \rightarrow \free r \rightarrow \sdp{\free r}{\free2} \rightarrow \free 2 \rightarrow 1$ is a short exact sequence such that $\sdp{\free r}{\free 2}$ is hyperbolic.
	The result follows from the main theorem (see Theorem \ref{intro:theo:main theorem}).
\end{proof}

\paragraph{} We are now looking for free abelian subgroups of $\out{\burn rn}$.
Let $G_1$ and $G_2$ be two torsion-free groups.
We denote by $G$ the free product $G_1 * G_2$.
Since $G_1$ and $G_2$ are torsion-free, so is $G$ (see \cite{Ser77}).

\begin{lemm}
\label{theo:root in free product}
	Let $g$ be a non-trivial element of $G$.
	If there exists a positive integer $k$ such that $g^k$ belongs to $G_1$, then $g$ belongs to $G_1$.
\end{lemm}

\begin{proof}
	We use the Bass-Serre theory of groups acting on trees (see \cite{Ser77}).
	There exist a simplicial tree $T$ and a simplicial action without inversion of $G$ on $T$ satisfying the following properties.
	The stabilizers of the vertices are the conjugates of $G_1$ and $G_2$.
	The stabilizers of the edges are trivial.
	Since $G$ is torsion-free, $g^k$ is a non-trivial element of $G_1$. 
	In particular it fixes a unique point $x$ of $T$: the one whose stabilizer is $G_1$.
	Therefore $g$ is an elliptic isometry and fixes a point of $T$ (see \cite[Chap. 9, Cor.~3.2]{CooDelPap90}).
	This last point has to be $x$, otherwise $g^k$ would fix two distinct points of $T$.
	Consequently $g$ belongs to the stabilizer of $x$, i.e. $G_1$.
\end{proof}

\begin{lemm}
	Let $n$ be an integer. The canonical map $G_1 \hookrightarrow G$ induces an injective homomorphism $j: G_1/G_1^n \hookrightarrow G/G^n$.
\end{lemm}

\begin{proof}
	By the previous lemma $G_1 \cap G^n = G_1^n$. Thus the kernel of the map $G_1 \rightarrow G \rightarrow G/G^n$ is exactly $G_1^n$.
\end{proof}

\begin{lemm}
\label{theo:injective burnside map}
	Let $n$ be an integer.
	Let $\phi$ be an automorphism of $G$ which stabilizes the factor $G_1$.
	We denote by $\phi_1$ the restriction of $\phi$ to $G_1$.
	If $\phi$ induces a finite order automorphism of $G/G^n$ then $\phi_1$ induces a finite order automorphism of $G_1/G_1^n$.
\end{lemm}

\begin{proof}
	We respectively denote by $\bar \phi_1$ and $\bar \phi$ the automorphisms of $G_1/G_1^n$ and $G/G^n$ induced by $\phi_1$ and $\phi$.
	By assumption, there exists an integer $k$ such that $\bar \phi^k = \id$.
	However, the following diagram is commutative.
	\begin{center}
		\begin{tikzpicture}[description/.style={fill=white,inner sep=2pt}] 
			\matrix (m) [matrix of math nodes, row sep=3em, column sep=2.5em, text height=1.5ex, text depth=0.25ex] 
			{ 
				G_1/G_1^n 	& G/G^n \\ 
				G_1/G_1^n		& G/G^n \\ 
			}; 
			\path[>=stealth, right hook->] 
			(m-1-1)	edge node[auto, above] {$j$} (m-1-2) 
			(m-2-1) 	edge node[auto, below] {$j$} (m-2-2);
			\path[>=stealth, ->] 
			(m-1-1)	edge node[auto, left] {$\bar \phi_1$} (m-2-1) 
			(m-1-2) 	edge node[auto, right] {$\bar \phi$} (m-2-2); 
		\end{tikzpicture} 
	\end{center}
	Thus $j\circ\bar \phi_1^k = \bar \phi^k \circ j = j$.
	Since $j$ is injective (see Lemma \ref{theo:injective burnside map}) $\bar\phi_1^k=\id$.
	In particular, $\bar \phi_1$ has finite order. 
\end{proof}

\begin{theo}
\label{theo:abelian free subgroups}
Let $r \geq 2$.
There exists an integer $n_0$ such that for all odd integers $n$ larger than $n_0$, $\out{\burn {2r}n}$ and $\out{\burn {2r+1}n}$ contain a subgroup which is isomorphic to $\Z^r$.
\end{theo}

\begin{proof}
	We denote by $\phi$ the automorphism of $\free 2$ studied in Proposition~\ref{theo:infinite order automorphism novikov}.
	There exists an integer $n_0$ such that for all odd integers $n$ larger than $n_0$, $\phi$ induces an infinite order automorphism of $\burn 2n$.
	We consider $\free {2r}$ as a free product $F_1*\dots *F_r$ of $r$ copies of $\free 2$.
	For all $i \in \intvald 1r$, we define an automorphism $\phi_i$ of $\free {2r}$ as follows.
	\begin{enumerate}
		\item The restriction of $\phi_i$ to $F_i$ is $\phi$.
		\item The restriction of $\phi_i$ to any other factor is the identity.
	\end{enumerate}
	We respectively denote by $\bar \phi_i$  and $\bar \phi$ the automorphisms of $\burn {2r}n$ and $\burn 2n$ induced by $\phi_i$ and $\phi$.
	By construction, the $\bar\phi_i$'s generate an abelian subgroup of $\aut {\burn {2r}n}$.
	We now study the relations between the $\bar \phi_i$'s in $\out{\burn {2r}n}$.
	Consider $r$ integers $k_1,\dots,k_r$ such that $\psi = \phi_1^{k_1}\dots \phi_r^{k_r}$ induces an inner automorphism $\bar \psi$ of $\burn{2r}n$.
	In particular, $\bar\psi$ has finite order.
	By Lemma \ref{theo:injective burnside map}, the restriction of $\psi$ to  $F_i$ induces a finite order automorphism of $F_i/F_i^n$.
	In other words, $\bar \phi^{k_i}$ has finite order.
	By construction, $\bar\phi$ has infinite order.
	This forces $k_i$ to be zero.
	There is hence no relation between the $\bar \phi_i$'s in $\out{\burn{2r}n}$.
	Thus the $\bar\phi_i$'s generate a subgroup of $\out{\burn{2r}n}$ which is isomorphic to $\Z^r$.
	For $\out{\burn{2r+1}n}$ we apply the same argument with the following free factorization:  $\free {2r+1} = F_1 * \dots * F_r * \Z$.
\end{proof}

%% file: 2_small_cancellation.tex
\section{Small cancellation theory}
\label{section:small cancellation theory}
\paragraph{}In this section, we expose the geometrical point of view on small cancellation developed by T. Delzant  and M. Gromov in \cite{DelGro08} and used in Section 3 to prove the main theorem.

\subsection{Hyperbolic spaces}

\paragraph{}Let $X$ be a proper, geodesic, $\delta$-hyperbolic (in the sense of Gromov) space.
The distance between two points $x$ and $x'$ of $X$ is denoted by $\dist[X]x{x'}$ (or simply $\dist x{x'}$).
Although it may not be unique, we denote by $\geo x{x'}$ a geodesic joining $x$ and $x'$.
The boundary at infinity of $X$ is denoted by $\partial X$ (see \cite[Chap. 2]{CooDelPap90}).
A part $Y$ of $X$ is $\alpha$-quasi-convex if every geodesic of $X$ joining two points of $Y$ lies in the $\alpha$-neighbourhood of $Y$, denoted by $Y^{+\alpha}$.

\begin{lemm}[see {\cite[Lemma 2.1.5]{DelGro08}} or {\cite[Cor. 1.2.2]{Cou09}}]
\label{theo:four points lemma}
	Let $x$, $x'$, $y$ and $y'$ be four points of $X$.
	Let $u$ be a point of $\geo x{x'}$ such that $\dist ux > \dist xy + 8 \delta$ and $\dist u{x'} > \dist {x'}{y'} + 8 \delta$.
	Then $u$ belongs to the $8 \delta$-neighbourhood of $\geo y{y'}$.
\end{lemm}

\begin{prop}[see {\cite[Lemma 2.2.2]{DelGro08}} or {\cite[Prop. 1.2.4]{Cou09}}]
\label{theo:intersection of quasi-convexes}
	Let $Y$ and $Z$ be two $\alpha$-quasi-convex parts of $X$. 
	For all $A>0$, we have
	\begin{displaymath}
		\diam \left(Y^{+A} \cap Z^{+A} \right) \leq \diam \left( Y^{+\alpha+10 \delta} \cap Z^{+\alpha+10 \delta} \right) + 2A + 20 \delta.
	\end{displaymath}
\end{prop}

\paragraph{}Let $G$ be a group acting properly, co-compactly, by isometries on $X$.
An element $g$ of $G$ is either \emph{elliptic} (in particular it has finite order) or \emph{hyperbolic} (see \cite[Chap. 9]{CooDelPap90}).
In the second case, $g$ fixes exactly two points of $\partial X$ denoted by $g^-$ and $g^+$.
In order to measure the action of an isometry $g$ of $X$, we define two translation lengths.
The \emph{translation length} $\len[espace=X]g$ (or simply $\len g$) is given by 
\begin{displaymath}
	\len g = \inf_{x \in X} \dist {gx}x.
\end{displaymath}
The \emph{asymptotic translation length} $\len[stable, espace=X] g$ (or simply $\len[stable]g$) is defined by 
\begin{displaymath}
	\len[stable] g = \lim_{n \rightarrow + \infty} \frac 1n \dist{g^nx}x.
\end{displaymath}
These two lengths satisfy the following inequality (see \cite[Chap. 10, Prop 6.4]{CooDelPap90}):
\begin{displaymath}
	\forall g \in G,\quad \len[stable]g \leq \len g \leq \len[stable] g + 32 \delta.
\end{displaymath}
An isometry of $X$ is hyperbolic if and only if its asymptotic translation length is positive (see \cite[Chap. 10, Prop. 6.3]{CooDelPap90}).
The \emph{axis} $A_g$ of an isometry $g$, defined as follows, is a $40 \delta$-subset of $X$ (cf. \cite[Prop. 2.3.3]{DelGro08}).
\begin{displaymath}
	A_g = \left\{ x \in X/ \dist {gx}x \leq \max\left\{ \len g, 40 \delta\right\} \right\}
\end{displaymath}

\begin{prop}
\label{theo:embedded quasi-convex in an axis}
	Let $g$ be a hyperbolic element of $G$.
	We denote by $\sigma$ a geodesic joining $g^-$ and $g^+$, the points of $\partial X$ fixed by $g$.
	Let $Y$ be a $\alpha$-quasi-convex part of $X$.
	If $Y$ is $g$-invariant, then $\sigma$ is contained in the $(\alpha+8 \delta)$-neighbourhood of $Y$.
	In particular $\sigma$ is contained in the $50 \delta$-neighbourhood of $A_g$.
\end{prop}

\begin{proof}
	Let $x$ be a point of $\sigma$.
	We denote by $d$ the distance from $x$ to $Y$ and by $y$ a point of $Y$ such that $\dist xy \leq d + \delta$.
	Since $g$ is hyperbolic, there exists an integer $m$ such that $\dist {g^mx}{g^{-m}x} \geq 2d + 100\delta$ (see \cite[Chap. 10, Lemme 6.5]{CooDelPap90}).
	We respectively denote by $p_-$ and $p_+$ the projections of $g^{-m}x$ and $g^mx$ on $\sigma$.
	The geodesics $\sigma$ and $g^m\sigma$ have the same extremities. 
	It follows that  they are $8\delta$-closed (see \cite[Chap. 2, Prop 2.2]{CooDelPap90}).
	In particular $\dist{g^mx}{p_+} \leq 8 \delta$.
	In the same way, we have $\dist{g^{-m}x}{p_-} \leq 8 \delta$.
	Note that $x$ lies on the subgeodesic of $\sigma$ delimited by $p_-$ and $p_+$.
	Indeed, if it was not the case, we should have
	\begin{eqnarray*}
		\dist{g^{-m}x}{g^mx} \leq \dist{p_-}{p_+} + 16 \delta
		& \leq & \left| \dist x{p_-} - \dist x{p_+}\right| + 16 \delta \\
		& \leq & \left| \dist x{g^{-m}x} - \dist x{\vphantom{g^{-m}}g^mx}\right| + 32 \delta \\
		& \leq & 32 \delta.
	\end{eqnarray*}
	Contradiction.
	On the other hand, we have
	\begin{displaymath}
		\dist x{p_+} \geq \dist x{g^mx} - 8 \delta \geq \frac 12 \dist{g^{-m}x}{g^mx} - 8 \delta \geq d + 30 \delta.
	\end{displaymath}
	In the same way, we have $\dist x{p_-} \geq d + 30 \delta$.
	By lemma \ref{theo:four points lemma}, $x$ belong to the $8 \delta$-neighbourhood of $\geo{g^{-m}y}{g^my}$.
	However $g^{-m}y$ and $g^my$ belongs to $Y$ which is $\alpha$-quasi-convex.
	Therefore the distance between $x$ and $Y$ is smaller than $\alpha + 8 \delta$.
\end{proof}

The \emph{injectivity radius} of a part $P$ of $G$ on $X$ is defined by 
\begin{displaymath}
\label{def:injectivity radius}
	\rinj P X = \inf\left\{ \len[stable]g / g \text{ is a hyperbolic element of } P\right\}.
\end{displaymath}

\paragraph{} A subgroup of $G$ is called \emph{elementary} if it is virtually cyclic.
Since $G$ is a hyperbolic group, any non-elementary subgroup of $G$ contains a copy of $\free 2$, the free group of rank 2 (see \cite[Chap. 8, Theo. 37]{GhyHar90}).
Given a hyperbolic isometry $g$ of $X$, the normalizer of $\left<g\right>$ is elementary (see \cite[Chap. 10, Cor. 7.2]{CooDelPap90}).

\paragraph{}The group $G$ satisfies the \emph{small centralizers hypothesis} if $G$ is non-elementary and any elementary subgroup of $G$ is cyclic.
\label{def:small centralizers hypothesis}
In order to study such a group we define the invariant $\Delta(G,X)$.
\label{def:invariant delta}
It is the upper bound of $\diam \left( A_g^{+50 \delta} \cap A_{g'}^{+50 \delta}\right)$, where $g$ and $g'$ are two elements of $G$ which generate a non-elementary subgroup and whose translation lengths are smaller than $100 \delta$ (see also \cite{DelGro08}).

\subsection{Small cancellation theorem}

\paragraph{} 
For the rest of Section \ref{section:small cancellation theory}, we assume that $X$ is simply-connected, and $G$ satisfies the small centralizers hypothesis.
Let $H$ be a normal subgroup of $G$ and $P$ a set of hyperbolic elements of $H$ stable by conjugation.
We also assume that $P$ only contains a finite number of conjugacy classes.
Let $N$ be the subgroup of $G$ generated by $P$. 
\paragraph{} 
The goal is to study the quotient $\bar G = G/N$.
To that end, we use a small cancellation assumption whose statement requires the following objects.
Let $\rho$ be an element of $P$.
We denote by $Y_\rho$ the set of points of $X$ which are $10 \delta$-closed to a geodesic joining $\rho^-$ and $\rho^+$.
The set $Y_\rho$ is $10 \delta$-quasi-convex (see \cite[Lemma 1.2.8]{Cou09}).
The subgroup of $G$ which stabilizes $Y_\rho$ is denoted by $E_\rho$. 
It is an elementary subgroup of $G$ (see \cite[Chap. 10, Prop. 7.1]{CooDelPap90}).
The parameters $\Delta(P)$ and $\rinjRF P$, defined below, respectively play the role of the length of the largest piece and the length of the smallest relation in the usual small cancellation theory.
\begin{eqnarray*}
	\Delta (P) & = & \sup_{\rho \neq \tau} \diam \left(Y_\rho^{+20 \delta} \cap Y_\tau^{+20 \delta}\right) \\
	\rinjRF P & = & \inf_{\rho \in P} \len[stable] \rho
\end{eqnarray*}

We are interested in situations where the ratios $\frac {\delta}{\rinjRF P}$ and $\frac {\Delta(P)}{\rinjRF P}$ are very small (see Theorem \ref{theo:small cancellation theorem} below).
We construct now a space $\bar X$ on which $\bar G$ acts properly, co-compactly by isometries.
Let $r_0$ be a positive number. 
Its value will be made precise in the small cancellation theorem (see Theorem \ref{theo:small cancellation theorem}).
Let $\rho \in P$.
We endow $Y_\rho$ with the length metric $\distV[\rho]$ induced by the restriction of $\distV[X]$ to $Y_\rho$.
The \emph{cone over $Y_\rho$} denoted by $C_\rho(r_0)$ (or simply $C_\rho$) is the topological quotient of $Y_\rho \times \intval 0{r_0}$ by the equivalence relation which identifies the points $(y,0)$, $y \in Y_\rho$.
If $(y,r)$ and $(y',r')$ are two points of $C_\rho$ the following formula defines a distance on $C_\rho$ (see \cite[Chap. I.5, Prop 5.9(1)]{BriHae99}):
\begin{displaymath}
	\cosh \left( \dist{\fantomB(y,r)}{(y',r')}\right) = \cosh r \cosh r' - \sinh r \sinh r' \cos \left( \min \left\{ \pi , \frac {\dist[\rho]y{y'}}{\sinh r_0}\right\}\right).
\end{displaymath}

\paragraph{} The \emph{cone-off over $X$ relatively to $P$}, denoted by $\dot X_P(r_0)$ (or simply $\dot X)$ is obtained by attaching, for all $\rho \in P$, the cone $C_\rho$ on $X$ along $Y_\rho$.
The distances $\distV[X]$ and $\distV[C_\rho]$ induce a metric on $\dot X$ (see \cite[Prop.~3.1.7]{Cou09}). 
We extend by homogeneity the action of $G$ on $X$ in an action of $G$ on $\dot X$:
if $x=(y,r)$ is a point of the cone $C_\rho$ and $g$ an element of $G$, then $gx$ is the point of the cone $C_{g\rho g^{-1}}=gC_\rho$ defined by $(gy,r)$.
Thus, $G$ acts by isometries on $\dot X$ (see \cite[Lemma 4.3.1]{Cou09})
The metric space $\bar X_P(r_0)$ (or simply $\bar X$) is the quotient of $\dot X$ by $N$.
It is a proper, geodesic, simply-connected metric space. 
Moreover $\bar G$ acts properly, co-compactly, by isometries on it (see \cite[Prop~II.3.12]{Cou10}).

\begin{theo}[Small cancellation theorem, see {\cite[Th.~5.5.2]{DelGro08}} or {\cite[Th.~4.2.2]{Cou09}}]
\label{theo:small cancellation theorem}
	There exist positive numbers $\delta_0$, $\delta_1$, $\Delta_0$ and $r_0\geq 10^5\delta_1$, that do not depend on $X$, $G$ or $P$ such that, if $\delta\leq \delta_0$, $\Delta(P) \leq \Delta_0$ and $\rinjRF P \geq 3 \pi \sinh r_0$, then the space $\bar X_P(r_0)$ is $\delta_1$-hyperbolic. 
	In particular, $\bar G$ is a hyperbolic group.
\end{theo}

\rem The fact that the constants $r_0$, $\delta_0$, $\delta_1$ and $\Delta_0$ do not depend on $X$, $P$ or $G$ is very important in order to iterate the small cancellation construction.

\subsection{Estimation of the injectivity radius of $\bar H$}

\paragraph{} 
We suppose now that the assumptions of the small cancellation theorem are fulfilled.
In order to iterate the construction, we need an estimation of the small cancellation parameters for $\bar G$.
This can be achieved by controlling the constants $\Delta(\bar G, \bar X)$ and $\rinj {\bar H} {\bar X}$, where $\bar H$ is the image of $H$ by the projection $\pi : G \rightarrow \bar G$.
Let $\nu$ be the canonical map $\nu : \dot X \rightarrow \bar X$.
The space $\bar X$ is obtained by gluing cones of large radius on $\nu (X)$.
This construction is  a kind of Margulis decomposition.
The cones play the role of the thick part: the translation length of a hyperbolic element of $\bar G$ on a cone is very large.
In particular we have the following lemma.

\begin{lemm}[see {\cite[Lemme 5.9.3]{DelGro08}}]
\label{theo:margulis thick part}
	Let $\bar g$ be an element of $\bar G$ such that $\len {\bar g} \leq 200 \delta_1$.
	Assume that, for all $ \rho \in P$, $\bar g$ does not belong to $\bar E_\rho = \pi \left( E_ \rho\right)$.
	Then $A_{\bar g}$ is contained in $\nu(X)^{+100\delta_1}$ and $A_{\bar g} \cap \nu(X)$ is non-empty.
\end{lemm}

To study $\nu(X)^{+100 \delta_1}$, which is an analogue of the thin part of the Margulis decomposition, we use the fact that the map $\nu(X) \rightarrow \bar X$ is a local quasi-isometry:

\begin{lemm}[see {\cite[Prop. 3.1.8]{Cou09}}]
\label{theo:margulis quasi-isometry}
	Let $x$ and $x'$ be two points of $X$.
	First, $\dist[\dot X] x{x'} \leq \dist[X] x{x'}$.
	Moreover, if $\dist[\dot X] x{x'} \leq \frac {r_0}2$ then $\dist[X] x{x'} \leq \frac {2 \pi \sinh r_0}{r_0} \dist[\dot X] x{x'}$.
\end{lemm}

Using this point of view, T. Delzant and M. Gromov proved the following result.

\begin{prop}[see {\cite[Lemme. 5.10.1]{DelGro08}}]
\label{theo:lifting quasi-convex}
	Let $\bar C$ be a $50 \delta_1$-quasi-convex part of $\nu(X)^{+100 \delta_1}$.
	There exists a part $C$ of $\dot X$ having the following properties:
	\begin{enumerate}
		\item the map $\nu : \dot X \rightarrow \bar X$ induces an isometry from $C$ onto $\bar C$,
		\item the projection $\pi : G \rightarrow \bar G$ induces an isomorphism from $\stab C$ onto $\stab {\bar C}$ which are respectively the stabilizers of $C$ and $\bar C$.
	\end{enumerate}
\end{prop}

\begin{prop}
\label{theo:estimation injectivity radius}
	The injectivity radius of $\bar H$ on $\bar X$ is bounded below by $\min \left\{ \kappa l, \delta_1\right\}$ where $l$ is the smallest asymptotic translation length of a hyperbolic element of $H$ that does not belong to any $E_\rho$ and $\kappa$ is equal to $\frac {r_0}{8\pi \sinh r_0}$.
\end{prop}

\rem This lemma improves Lemma~5.11.1 proved by T.~Delzant and M.~Gromov in \cite{DelGro08}.
They gave indeed a lower bound for $\rinj {\bar G}{\bar X}$.
Four our purpose, we need a more accurate result.
We propose here an estimation of the injectivity radius of a normal subgroup of $\bar G$.

\begin{proof}
	Let $m$ be the largest integer such that $m \min \left\{ \kappa l , \delta_1\right\} \leq 40 \delta_1$.
	Let $\bar h$ be a hyperbolic element of $\bar H$.
	We assume that $\len {\bar h ^m } \leq m \min \left\{ \kappa l, \delta_1\right\} + 40 \delta_1$.
	We denote by $\bar C$, the axis of $\bar h^m$ in $\bar X$, which is $50 \delta_1$-quasi-convex (see \cite[Prop. 2.3.3]{DelGro08}).
	Since $\len{\bar h^m} \leq 80 \delta_1$, the axis $\bar C$ is contained in the $100\delta_1$ neighbourhood of  $\nu (X)$ (see Lemma \ref{theo:margulis thick part}).
	By Proposition \ref{theo:lifting quasi-convex}, there exists a part $C$ of $X$ such that
	\begin{enumerate}
		\item the map $\nu : \dot X \rightarrow \bar X$ induces an isometry from $C$ onto $\bar C$,
		\item the map $\pi : G \rightarrow \bar G$ induces an isomorphism from $\stab C$ onto $\stab {\bar C}$.
	\end{enumerate}
	However $\bar h$ belongs to $\stab {\bar C}$.
	We denote by $h$ its preimage in $\stab C$.
	Since $\bar h$ is hyperbolic, $h$ is necessarily hyperbolic and does not belong to any $E_\rho$, $\rho \in P$.
	Note that the relations $P$ are contained in $H$.
	Thus $N$ lies in $H$.
	It follows that $h$ is an element of $H$.
	Hence, by assumption, $\len[stable, espace=X] h \geq l$.
	
	\paragraph{} On the other hand, by Lemma \ref{theo:margulis thick part}, the intersection $\bar C \cap \nu(X)$ is non-empty.
	We chose a point $\bar x$ in $\bar C \cap \nu(X)$ and denote by  $x$ its preimage in $C$.
	The map $\nu : C \rightarrow \bar C$ is an equivariant isometry, thus we have
	\begin{displaymath}
		\dist[\dot X]{h^mx} x = \dist[\bar X]{\bar h^m\bar x}{\bar x} \leq \max\left\{ \len[stable] {\bar h^m}, 40 \delta_1\right\} \leq 80 \delta_1.
	\end{displaymath}
	By Lemma \ref{theo:margulis quasi-isometry}, $\dist[X]{h^mx}x$ is smaller than $\frac{2\pi \sinh r_0}{r_0}\dist[\dot X]{h^mx}x$.
	Consequently,
	\begin{displaymath}
		ml \leq m \len[stable, espace=X] h \leq \len[espace=X]{h^m} \leq \dist[X]{h^mx}x \leq \frac{160 \pi \delta_1 \sinh r_0}{r_0} = \frac { 20\delta_1}\kappa .
	\end{displaymath}
	In particular $m\kappa l \leq 20 \delta_1$.
	It follows that $m$ is not the largest integer such that $m \min \left\{ \kappa l, \delta_1\right\} \leq 40 \delta_1$.
	Contradiction.
	Therefore, $\len{\bar h ^m}$ is larger than $m \min \left\{ \kappa l, \delta_1\right\} + 40 \delta_1$.
	We now use the inequality linking asymptotic translation lengths and translation lengths (see \cite[Chap.~10, Prop.~6.4]{CooDelPap90}):
	\begin{displaymath}
	 	m \len[stable]{\bar h} = \len[stable]{\bar h^m} \geq \len{\bar h^m} - 32 \delta_1 \geq m \min \left\{ \kappa l, \delta_1\right\}.
	\end{displaymath}
	This last inequality holds for any hyperbolic element $\bar h$ in $\bar H$.
	Thus $\rinj {\bar H}{\bar X}$ is larger than $\min\left\{ \kappa l, \delta_1\right\}$.
\end{proof}

\subsection{Other properties of $\bar G$ and $\bar X$}

We recall here some results obtained by T. Delzant and M. Gromov in \cite{DelGro08}.

\begin{prop}[see {\cite[Lemme 5.9.5]{DelGro08}}]
\label{theo:majoration Delta}
	The constant $\Delta(\bar G, \bar X)$ satisfies the following inequality:
	\begin{displaymath}
		\Delta(\bar G, \bar X) \leq \Delta (G, X) + 1000 \delta_1 e^{350 \delta_1}.
	\end{displaymath}
\end{prop}

\begin{prop}[see {\cite[Lemme 5.10.2 and Lemme 5.10.3]{DelGro08}}]
\label{theo:finite subgroups}
	Assume that any element of $P$ is an odd power of an element of $G$ which is not a proper power.
	Then $\bar G$  satisfies the following properties:
	\begin{enumerate}
		\item every elementary subgroup of $\bar G$ is cyclic,
		\item let $\bar F$ be a finite subgroup of $\bar G$. Either $\bar F$ is the image of a finite subgroup of $G$, or there exists $\rho \in P$ such that $\bar F$ is a subgroup of $\bar E_\rho = \pi(E_ \rho)$.
	\end{enumerate}
\end{prop}

\begin{prop}[see {\cite[Th. 5.7.1]{DelGro08}}]
\label{theo:euler characterstic}
	The Euler characteristic of $\bar G$ satisfies $\chi(\bar G, \Q) = \chi (G, \Q) + \left| P/G \right|$, where $\left|P/G\right|$ denotes the number of conjugacy classes of $P$.
\end{prop}

\subsection{An induction lemma}

\paragraph{} The following lemma must be seen as the fundamental step of the induction which will be used to prove the main theorem.
We recall that the invariant $\Delta(G,X)$ represents the maximal overlap between the axis of two small hyperbolic elements (see p.\pageref{def:invariant delta}). 
The injectivity radius $\rinj HX$ denotes the smallest asymptotic translation length of a hyperbolic element of $H$ (see p.\pageref{def:injectivity radius}).
\begin{lemm}[Induction lemma]
\label{induction lemma}
There exist positive numbers $\delta_1$, $\Delta_1$, $l_1$, $l_2$, $l_3$ and an integer $n_0$ satisfying the following properties.
Let $n$ be an odd integer larger than $n_0$. 
Let $X$ be a proper, geodesic, simply-connected, $\delta_1$-hyperbolic space.
Let $G$ be a group acting properly, co-compactly, by isometries on $X$ and $H$ a normal subgroup of $G$ such that
\begin{enumerate}
	\item \label{enumerate:small centralizers}
	$G$ satisfies the small centralizers hypothesis and the order of every finite subgroup of $G$ divides $n$,
	\item \label{enumerate:small cancellation assumptions}
	$\Delta(G,X) \leq \Delta_1$ and $\rinj HX \geq \frac {l_2}{\sqrt n}$.
\end{enumerate}

We denote by $R$ the set of hyperbolic elements of $H$, which are not a proper powers in $G$ and whose asymptotic translation lengths are smaller than $l_1$.
Let $N$ be the normal subgroup of $G$ generated by $\left\{h^n/h \in R\right\}$, $\bar G$ the quotient group $G/N$ and $\bar H$ the image of $H$ by the canonical map $\pi : G \rightarrow \bar G$.
We assume that $\chi\left(G, \Q\right) + \frac 12\left|R/G\right| >0$, where $\chi\left(G,\Q\right)$ is the Euler characteristic of $G$ and $\left| R/G\right|$ the number of conjugacy classes in $R$.

\paragraph{}Then, there exists a proper, geodesic, simply-connected, $\delta_1$-hyperbolic space $\bar X$ on which $\bar G$ acts properly, co-compactly, by isometries. 
Moreover $\bar G$, $\bar H$ and $\bar X$ satisfy the points \ref{enumerate:small centralizers} and \ref{enumerate:small cancellation assumptions} 
Furthermore, $\Delta(\bar G, \bar X)>0$ and 
\begin{displaymath}
	\forall g \in G, \quad \len[stable, espace={\bar X}] {\pi(g)} \leq \frac {l_3}{\sqrt n} \len[stable, espace=X] g.
\end{displaymath}
\end{lemm}

\rem If $G$, $H$, $X$ and $n$ satisfy the hypothesis of the previous lemma we will say that $(G,H,X)$ \textit{satisfies the induction assumptions for exponent $n$}.
The Induction lemma says in particular that if $(G,H,X)$ satisfies the induction assumptions for exponent $n$, so does $(\bar G, \bar H, \bar X)$.
\begin{proof}

The positive constants $r_0$, $\delta_0$, $\delta_1$, and $\Delta_0$ are given by the small cancellation theorem (see Theorem \ref{theo:small cancellation theorem}).
The constant $\kappa = \frac {r_0}{8\pi \sinh r_0}$ is the one that appears in Proposition \ref{theo:estimation injectivity radius}.
We define a renormalization parameter $L_n = \sqrt{\frac{n\kappa \delta_1}{\pi \sinh r_0}}$.
The sequence $(L_n)$ is increasing and tends to infinity.
Up to chose $n_0$ large enough, we may assume that for all $n \geq n_0$,
\begin{eqnarray*}
	\frac{2000 \delta_1 e^{350 \delta_1} +300 \delta_1}{L_n} & \leq & \min \left\{ \Delta_0, 1000\delta_1 e^{350 \delta_1}\right\}, \\
	\frac{\delta_1}{L_n} & \leq & \delta_0, \\
	\frac{3 \kappa \delta_1}{L_n} & \leq & \delta_1.
\end{eqnarray*}
Note that $n_0$ only depend on $\delta_1$ and $r_0$.
We now define the following constants:
\begin{displaymath}
	\Delta_1 = 2000 \delta_1e^{350 \delta_1}, \quad 
	l_1 = 3 \delta_1, \quad 
	l_2 = 3\sqrt{\kappa \delta_1 \pi \sinh r_0 } \text{ and } 
	l_3 =   \sqrt{\frac{\pi \sinh r_0}{\kappa \delta_1}}.
\end{displaymath}

Let $n$ be an odd integer larger than $n_0$. 
We assume that $(G, H,X)$ satisfies the induction assumptions for exponent $n$.
In particular, $R$ is the set of hyperbolic elements of $H$, which are not a proper powers in $G$ and whose asymptotic translation lengths are smaller than $l_1$.

\begin{lemm}
\label{theo:subset without inverse}
	There exists a subset $R_0$ of $R$, stable by conjugation, such that for all $h \in R$, exactly one of the elements $h$ or $h^{-1}$ belongs to $R_0$. 
\end{lemm}

\begin{proof}
	It is sufficient to prove that an element $h$ of $R$ cannot be conjugate to its inverse.
	Assume that this fact is false.
	There exist $h \in R$ and $g \in G$ such that $ghg^{-1} = h^{-1}$.
	Therefore, $g$ belongs to the normalizer of $h$, which is elementary (see \cite[Chap. 10, Prop. 7.1]{CooDelPap90}).
	In particular $g$ and $h$ generate an elementary subgroup of $G$.
	Thus $g$ and $h$ commute. 
	It follows that $h = h^{-1}$.
	In particular, $h$ is not hyperbolic.
	Contradiction.
\end{proof}

We now study the set of relations $P = \left\{ h^n, h \in R_0\right\}$.
To that end, we consider the action of $G$ on the renormalized space $\frac 1{L_n}X$, which is $\delta_0$-hyperbolic.
\begin{lemm}
\label{theo:sc assumptions satisfied}
The set $P$ satisfies the small cancellation assumptions of Theorem \ref{theo:small cancellation theorem}.
\end{lemm}

\begin{proof}
	First note that $\frac 1{L_n}X$ is $\delta_0$-hyperbolic.
	Let $h_1$ and $h_2$ be two elements of $R_0$ such that $h_1^n \neq h_2^n$.
	By Proposition \ref{theo:embedded quasi-convex in an axis}, $Y_{h_1^n}$ is contained in $A_{h_1}^{+100 \delta_0}$.
	Hence Proposition \ref{theo:intersection of quasi-convexes} gives
	\begin{eqnarray*}
		\diam \left( Y_{h_1^n}^{+20 \delta_0} \cap Y_{h_2^n}^{+20 \delta_0}\right) & \leq & \diam \left(A_{h_1}^{+120 \delta_0}\cap A_{h_2}^{+120 \delta_0}\right) \\
		& \leq & \diam \left( A_{h_1}^{+50 \delta_0} \cap A_{h_2}^{+50 \delta_0}\right) + 300 \delta_0.
	\end{eqnarray*}
	Assume that $h_1$ and $h_2$ generate an elementary subgroup.
	Since $G$ satisfies the small centralizers hypothesis, this subgroup should be cyclic.
	However $h_1$ and $h_2$ are not proper powers.
	Thus they are either equal or inverse.
	By construction of $R_0$, they cannot be inverse. 
	Thus $h_1 = h_2$, and \emph{a fortiori} $h_1^n=h_2^n$.
	Contraction.
	Consequently, $h_1$ and $h_2$ generate a non-elementary subgroup.
	In the other hand, $\len[stable]{h_1}$ and $\len[stable]{h_2}$ are smaller than $3 \delta_0$.
	By definition of $\Delta(G,X)$, we have
	\begin{eqnarray*}
		\diam \left( Y_{h_1^n}^{+20 \delta_0} \cap Y_{h_2^n}^{+20 \delta_0}\right) 
		& \leq & \Delta\left(G, \frac 1 {L_n}X \right) + 300 \delta_0 \\
		& \leq & \frac{\Delta_1 + 300 \delta_1}{L_n} = \frac {2000\delta_1e^{350 \delta_1}+ 300 \delta_1}{L_n}.
	\end{eqnarray*}
	Hence $\Delta(P)$ is smaller than $\Delta_0$.
	
	\paragraph{} The injectivity radius of $H$ on $\frac 1{L_n}X$ is larger than
	\begin{displaymath}
		\frac 1{L_n} \frac {l_2}{\sqrt n} =  \sqrt{\frac {\pi \sinh r_0}{n\kappa \delta_1}}\frac {3 \sqrt{\kappa \delta_1 \pi \sinh r_0}}{\sqrt n} = \frac{3\pi \sinh r_0}n.
	\end{displaymath}
	In particular for all $h \in R_0$, $\len[stable]{h^n} \geq 3 \pi \sinh r_0$.
	Therefore, $\rinjRF P \geq 3\pi\sinh r_0$. 
\end{proof}

Applying the small cancellation theorem, the space $\bar X = \bar X_P(r_0)$ is proper, geodesic, simply-connected, $\delta_1$-hyperbolic and $\bar G = G/ \ll P \gg$ acts properly, co-compactly, by isometries on it.

\begin{lemm}
	Every elementary subgroup of $\bar G$ is cyclic, either infinite or finite with order dividing $n$.
\end{lemm}

\begin{proof}
	All elements of $P$ are odd powers of elements of $G$ which are not proper powers.
	By Proposition \ref{theo:finite subgroups}, all elementary subgroups of $\bar G$ are cyclic.
	Assume now, that $\bar F$ is a finite subgroup of $\bar G$.
	Applying the same proposition, we distinguish two cases.
	\begin{enumerate}
		\item $\bar F$ is the image of a finite subgroup of $G$. 
		However, the order of every finite subgroup of $G$ divides $n$.
		Thus the order of $\bar F$ divides $n$.
		\item There exists $h \in R_0$ such that $\bar F$ is a subgroup of $\bar E_{h^n} = \pi\left(E_{h^n}\right) = \left< \pi(h)\right>$, whose order divides $n$.
	\end{enumerate}
\end{proof}

\begin{lemm}
	The constant $\Delta(\bar G, \bar X)$ is bounded above by $\Delta_1$.
	The injectivity radius $\rinj {\bar H}{\bar X}$ is bounded below by $\frac{l_2}{\sqrt n}$.
\end{lemm}

\begin{proof}
	By Proposition \ref{theo:majoration Delta}, $\Delta (\bar G, \bar X) \leq \Delta \left( G, \frac 1{L_n} X \right) +1000\delta_1e^{350 \delta_1}$.
	However, we assumed that
	\begin{displaymath}
		\Delta \left(G, \frac 1{L_n}X\right) = \frac 1{L_n} \Delta(G,X) \leq \frac {\Delta_1}{L_n} = \frac{2000\delta_1e^{350\delta_1}}{L_n} \leq 1000 \delta_1e^{350 \delta_1}.
	\end{displaymath}
	Hence $\Delta(\bar G, \bar X) \leq 2000 \delta_1 e^{350 \delta_1} = \Delta_1$.
	
	\paragraph{}Let $g$ be a hyperbolic element of $H$, which does not belong to any subgroup $E_{h^n} = \left < h\right>$, $h \in R_0$.
	Its asymptotic translation length in $\frac 1{L_n}X$ is larger than $\frac {l_1}{L_n} = \frac{3 \delta_1}{L_n}$.
	By Proposition \ref{theo:estimation injectivity radius}, we have
	\begin{displaymath}
		\rinj{\bar H}{\bar X} \geq \min \left\{ \frac {3\kappa \delta_1}{L_n} , \delta_1 \right\} =  \frac {3\kappa \delta_1}{L_n}  = \frac{3\sqrt{\kappa \delta_1\pi \sinh r_0}}{\sqrt n} =\frac{l_2}{\sqrt n}.
	\end{displaymath}
\end{proof}

\begin{lemm}
	The Euler characteristic of $\bar G$ satisfies $\chi(\bar G, \Q) = \chi(G, \Q) + \frac 12 \left|R/G \right|$, where $\left|R/G\right|$ is the number of conjugacy classes of $R$.
	In particular $\bar G$ is non-elementary.
\end{lemm}

\begin{proof}
	By construction, there are twice more conjugacy classes in $R$ than in $R_0$. 
	The result follows from Proposition \ref{theo:euler characterstic}.
\end{proof}

\begin{lemm}
	For all $g \in G$, we have $\len[stable, espace=\bar X]{\pi (g)} \leq \frac {l_3}{\sqrt n} \len[stable, espace=X] g$.
\end{lemm}

\begin{proof}
	By Lemma \ref{theo:margulis quasi-isometry}, the map $\frac 1{L_n}X \rightarrow \bar X$ contracts the distances.
	Thus for all $g \in G$, 
	\begin{displaymath}
		\len[stable, espace=\bar X] {\pi(g)} \leq \frac 1{L_n} \len[stable, espace=X] g =  \sqrt{\frac {\pi \sinh r_0}{n \kappa \delta_1}}\len[stable, espace=X] g = \frac{l_3}{\sqrt n} \len[stable, espace=X] g.
	\end{displaymath}
\end{proof}

\paragraph{} This last lemma ends the proof of the induction lemma.
\end{proof}
	

%% file: 3_main_theorem.tex
\section{Proof of the main theorem}
\label{sec:proof main theorem}

Recall the statement of the main theorem.
\begin{theo}
	Let $1 \rightarrow H \rightarrow G \rightarrow F \rightarrow 1$ be a short exact sequence of groups.
	Assume that $H$ is finitely generated, $G$ is torsion-free, hyperbolic and $F$ is torsion-free.
	Then there exists an integer $n_0$ such that for all odd integers $n$ larger than $n_0$, the canonical map $F \rightarrow \out H$ induces an injective homomorphism $F \hookrightarrow \out{H/H^n}$.
\end{theo}

\begin{proof}

The constants $\delta_1$, $\Delta_1$, $l_1$, $l_2$, $l_3$ and $n_0$ are given by the Induction lemma (see Lemma \ref{induction lemma}).
Up to increase $n_0$, we may also assume that $\frac {l_3}{\sqrt{n_0}} < 1$.
Let $1 \rightarrow H \rightarrow G \rightarrow F \rightarrow 1$ be a short exact sequence of groups, which satisfies the hypotheses of the theorem.
The strategy is to build by induction a family of short exact sequences ${1 \rightarrow H_k \rightarrow G_k \rightarrow F \rightarrow 1}$ with an action of $G_k$ on a hyperbolic space $X_k$, such that the direct limit $\dlim H_k$ is the Burnside group $H/H^n$.

\paragraph{Initialization.} We put $H_0 = H$, and $G_0 = G$.
Let $X_0$ be a proper, geodesic, simply-connected, hyperbolic space on which $G$ acts properly, co-compactly, by isometries.
Take for instance the Rips' polyhedron of $G$ (see \cite[Chap 5.]{CooDelPap90}).
Up to renormalize $X_0$, we may assume that:
\begin{itemize}
	\item $X_0$ is $\delta_1$-hyperbolic,
	\item $\Delta\left(G_0,X_0\right) \leq \Delta_1$,
	\item $\left\{h \in H / \len[stable] h \leq l_1, h \text{ is not a proper power }\right\}$ contains a number of conjugacy classes bounded below by $2 \chi\left(G, \Q\right)$.
\end{itemize}
Since $G$ is a hyperbolic group, the injectivity radius of $H$ is positive (see \cite{Del96}).
Thus, up to increase one more time $n_0$, we may assume that $\rinj {H_0}{X_0} \geq \frac{l_2}{\sqrt {n_0}}$.
It follows that $\left( G_0, H_0, X_0\right)$ satisfies the induction assumptions for exponent $n_0$.

\paragraph{} Let $n$ be an odd integer larger than $n_0$. 
$\left( G_0, H_0, X_0\right)$ satisfies \emph{a fortiori} the induction assumptions for exponent $n$.

\paragraph{Induction.} Let $(G_k, H_k, X_k)$ satisfying the induction assumptions for exponent $n$.
We denote by $R_k$ the set of hyperbolic elements of $H_k$ which are not proper powers in $G_k$ and whose asymptotic translation lengths are smaller than $l_1$.
Let $N_k$ be the normal subgroup of $G_k$ generated by $\left\{ h^n / h \in R_k\right\}$, $G_{k+1}$ the quotient $G_k/N_k$ and $H_{k+1}$ the image of $H_k$ by the canonical map $\pi_k : G_k \rightarrow G_{k+1}$.
By the Induction lemma, there exists a metric space $X_{k+1}$ such that $\left(G_{k+1}, H_{k+1}, X_{k+1}\right)$ satisfies the induction assumptions for the exponent $n$.
In this way, we obtain two sequences of groups $(H_k)$ and $(G_k)$ whose properties we want to study now.

\paragraph{Properties of $H_k$ and $G_k$.}

\begin{lemm}
\label{theo:short exact sequence}
	For all integers $k$, there exists a map $G_k \rightarrow F$ such that the following diagram is commutative.
	Moreover its rows are short exact sequences.
	\begin{center}
			\begin{tikzpicture}[description/.style={fill=white,inner sep=2pt},] 
				\matrix (m) [matrix of math nodes, row sep=2em, column sep=2.5em, text height=1.5ex, text depth=0.25ex] 
				{ 
					1	& H	& G	& F	& 1	\\
					1	& H_k	& G_k	& F	& 1	\\
				}; 
				\draw[>=stealth, ->] (m-1-1)	-- (m-1-2);
				\draw[>=stealth, ->] (m-1-2) -- (m-1-3);
				\draw[>=stealth, ->] (m-1-3) -- (m-1-4);
				\draw[>=stealth, ->] (m-1-4) -- (m-1-5);
				\draw[>=stealth, ->] (m-2-1) -- (m-2-2);
				\draw[>=stealth, ->] (m-2-2) -- (m-2-3);
				\draw[>=stealth, ->] (m-2-3) -- (m-2-4);
				\draw[>=stealth, ->] (m-2-4) -- (m-2-5);
				\draw[>=stealth, ->] (m-1-2)	-- (m-2-2);
				\draw[>=stealth, ->] (m-1-3)	-- (m-2-3);
				\draw[double distance=1.5pt] (m-1-4) -- (m-2-4);
			\end{tikzpicture} 
		\end{center}
\end{lemm}

\begin{proof}
	Following the construction of the groups $H_k$ and $G_k$, we prove this lemma by induction on $k$.
	The result is obvious for $k=0$.
	Consider now an integer $k$ for which the lemma holds.
	The subgroup $N_k$ is generated by elements of $H_k$.
	It follows that $N_k$ is contained in $H_k$, which is also the kernel of the map $G_k \rightarrow F$.
	Hence $G_k \rightarrow F$ induces a map from $G_{k+1} = G_k / N_k$ to $F$ such that the following diagram is commutative.
	\begin{center}
		\begin{tikzpicture}[description/.style={fill=white,inner sep=2pt},] 
			\matrix (m) [matrix of math nodes, row sep=1em, column sep=3em, text height=1.5ex, text depth=0.25ex] 
				{ 
					G_k		& 		\\
								&  F	\\
					G_{k+1}	&		\\
				}; 
			\path[>=stealth, ->] 
			(m-1-1)		edge node[auto] {} (m-2-2) 
			(m-3-1) 		edge node[auto] {} (m-2-2)
			(m-1-1)		edge node[auto, left] {$\pi_k$} (m-3-1);
		\end{tikzpicture} 
	\end{center}
	By definition, $H_{k+1}$ is the image of $H_k$ by the projection $\pi_k$.
	Since $\pi_k$ is onto, $H_{k+1}$ is the kernel of the map $G_{k+1} \rightarrow F$.
	Consequently, the following diagram commutes and its rows are short exact sequences.
	\begin{center}
		\begin{tikzpicture}[description/.style={fill=white,inner sep=2pt},] 
			\matrix (m) [matrix of math nodes, row sep=2em, column sep=2.5em, text height=1.5ex, text depth=0.25ex] 
			{ 
				1	& H_k			& G_k			& F	& 1	\\
				1	& H_{k+1}	& G_{k+1}	& F	& 1	\\
			}; 
			\draw[>=stealth, ->] (m-1-1)	-- (m-1-2);
			\draw[>=stealth, ->] (m-1-2) -- (m-1-3);
			\draw[>=stealth, ->] (m-1-3) -- (m-1-4);
			\draw[>=stealth, ->] (m-1-4) -- (m-1-5);
			\draw[>=stealth, ->] (m-2-1) -- (m-2-2);
			\draw[>=stealth, ->] (m-2-2) -- (m-2-3);
			\draw[>=stealth, ->] (m-2-3) -- (m-2-4);
			\draw[>=stealth, ->] (m-2-4) -- (m-2-5);
			\draw[>=stealth, ->] (m-1-2)	-- (m-2-2);
			\draw[>=stealth, ->] (m-1-3)	-- (m-2-3) node[midway, left]{$\pi_k$};
			\draw[double distance=1.5pt] (m-1-4) -- (m-2-4);
		\end{tikzpicture} 
	\end{center}
	Thus the lemma holds for $k+1$.
\end{proof}

We would like now to compare the groups $H/H^n$ and $\dlim H_k$.
For notational convenience, we will denote by $h$ an element of $H$ as well as its images in $H_k$, $\dlim H_k$ or $H/H^n$.

\begin{lemm}
\label{theo: kernel map burnside to dlim}
	The kernel of the canonical map $H \rightarrow \dlim H_k$ is exactly $H^n$, the subgroup of $H$ generated by all $n$-th powers.
\end{lemm}

\begin{proof}
	Let $h$ be an element of $H \setminus \{ 1 \}$.
	For all integers $k$, we have 
	\begin{math}
		\len[stable, espace={X_k}] h 
		\leq \left( \frac {l_3}{\sqrt n}\right)^k \len[stable, espace={X_0}] h 
		\leq \left( \frac {l_3}{\sqrt{n_0}}\right)^k \len[stable, espace={X_0}] h
	\end{math}
	(see Lemma  \ref{induction lemma}).
	However, we chose $n_0$ in such a way that $\frac {l_3}{\sqrt{n_0}}<1$.
	It follows that there exists an integer $k$ such that $\len[stable, espace={X_k}] h  < \frac {l_2}{\sqrt n}$. 
	The group $G_k$  satisfies the small centralizers hypothesis (i.e. $G_k$ is non-elementary and its elementary subgroups are cyclic).
	Thus there exists an element $r$ of $G_k$, which is not a proper power, and a positive integer $m$ such that $h = r^m$.
	In particular $r^m$ belongs to the kernel of the map $G_k \rightarrow F$.
	Since $F$ is torsion-free, $r$ also belongs to $H_k$.
	Note that its asymptotic length in $X_k$ is smaller than $\frac {l_2}{\sqrt n}$. 
	By construction, the injectivity radius of $H_k$ on $X_k$ is larger than $\frac {l_2}{\sqrt n}$ (point \ref{enumerate:small cancellation assumptions} of Lemma \ref{induction lemma}). 
	Therefor $r$ is an elliptic isometry. 
	In particular, it has finite order dividing $n$ (point \ref{enumerate:small centralizers} of Lemma \ref{induction lemma}).
	It follows that the image of $h^n$ in $H_k$ is trivial.
	Hence $H^n$ is contained in the kernel of $H \rightarrow \dlim H_k$.
	
	\paragraph{} On the other hand, at each step of the construction, the kernel of the map $H_k \rightarrow H_{k+1}$ is generated by $n$-th powers of elements of $H_k$.
	It follows that the kernel of the morphism $H \rightarrow \dlim H_k$ is contained in $H^n$.
\end{proof}

\begin{lemm}
\label{theo: isomorphism burnside and dlim}
	The groups $H/H^n$ and $\dlim H_k$ are isomorphic.
\end{lemm}

\begin{proof}
	The map $H \rightarrow \dlim H_k$ is onto. 
	Thanks to Lemma \ref{theo: kernel map burnside to dlim}, its kernel is $H^n$. 
	It follows that it induces an isomorphism between $H/H^n$ and $\dlim H_k$.
\end{proof}

\begin{lemm}
\label{theo:injection F out burn}
	Let $f$ be a non trivial element of $F$.
	Let $g$ be a preimage of $f$ by the map $G \rightarrow F$.
	The conjugation by $g$ defines an automorphism of $H$ which induces a non trivial outer automorphism of $H/H^n$.
\end{lemm}

\begin{proof}
	Let $S$ be a finite generating set of $H$.
	We denote by $\phi$ the automorphism of $H$ defined as follows: for all $h \in H$, $\phi(h) = ghg^{-1}$.
	Assume that $\phi$ induces an inner automorphism of $H/H^n$.
	There exists $l \in H$ such that for all $h \in H$, $\phi(h)$ and $lhl^{-1}$ have the same image in $H/H^n$.
	We proved previously that $H/H^n$ and $\dlim H_k$ are isomorphic (see Lemma \ref{theo: isomorphism burnside and dlim}).
	Since $S$ is finite, there exists an integer $k$ such that, for all $s \in S$, $\phi(s)$ and $lsl^{-1}$ are equal in $H_k$.
	However $S$ is a generating set of $H$.
	Thus for all $h \in H$, $\phi(h)=ghg^{-1}$ and $lhl^{-1}$ are equal in $H_k$.
	We use now the following commutative diagram (see Lemma \ref{theo:short exact sequence}).
	\begin{center}
		\begin{tikzpicture}[description/.style={fill=white,inner sep=2pt},] 
			\matrix (m) [matrix of math nodes, row sep=2em, column sep=2.5em, text height=1.5ex, text depth=0.25ex] 
			{ 	
				1	& H	& G	& F	& 1	\\
				1	& H_k	& G_k	& F	& 1	\\
			}; 
			\draw[>=stealth, ->] (m-1-1)	-- (m-1-2);
			\draw[>=stealth, ->] (m-1-2) -- (m-1-3);	
			\draw[>=stealth, ->] (m-1-3) -- (m-1-4);
			\draw[>=stealth, ->] (m-1-4) -- (m-1-5);
			\draw[>=stealth, ->] (m-2-1) -- (m-2-2);
			\draw[>=stealth, ->] (m-2-2) -- (m-2-3);
			\draw[>=stealth, ->] (m-2-3) -- (m-2-4);
			\draw[>=stealth, ->] (m-2-4) -- (m-2-5);
			\draw[>=stealth, ->] (m-1-2)	-- (m-2-2);
			\draw[>=stealth, ->] (m-1-3)	-- (m-2-3);
			\draw[double distance=1.5pt] (m-1-4) -- (m-2-4);
		\end{tikzpicture} 
	\end{center}
	The image of $l^{-1}g$ in $G_k$ commutes with every element of $H_k$.
	The subgroup $H_k$ is a normal subgroup of $G_k$ which satisfies the small centralizers hypothesis.
	Therefore, $H_k$ is non-elementary.
	In particular, it contains a hyperbolic element $h$.
	Hence $h$ and $l^{-1}g$ generates an abelian subgroup of $G_k$ which has to be cyclic.
	There exists $(p,q) \in \Z^*\times \Z$ such that $\left(l^{-1}g\right)^p = h^q$ in $G_k$. 
	Using one more time the commutative diagram, we push this identity in $F$ and obtain $f^p = 1$.
	Since $F$ is torsion-free, $f$ is trivial.
	Contradiction.
\end{proof}

\paragraph{End of the proof of the main theorem.} The map $F \rightarrow \out H$ can be constructed as follows.
Let $f$ be an element of $F$ and $g$ a preimage of $f$ by $G \rightarrow F$.
The image of $f$ by the map $F \rightarrow \out H$ is exactly the outer automorphism of $H$ induced by the conjugation by $g$ in $G$.
The previous lemma is hence a reformulation of the following fact: the map $F \rightarrow \out H$ induces an injective homomorphism $F \hookrightarrow \out{H/H^n}$.
This remark ends the proof of the main theorem.
\end{proof}

%% file: automorphisms.bbl
\begin{thebibliography}{LvdW33}

\bibitem[Adi79]{Adi79}
S.~I. Adian.
\newblock {\em The {B}urnside problem and identities in groups}, volume~95 of
  {\em Ergebnisse der Mathematik und ihrer Grenzgebiete [Results in Mathematics
  and Related Areas]}.
\newblock Springer-Verlag, Berlin, 1979.
\newblock Translated from the Russian by John Lennox and James Wiegold.

\bibitem[AL92]{AdiLys92}
S.~I. Adian and I.~G. Lysenok.
\newblock The method of classification of periodic words and the {B}urnside
  problem.
\newblock In {\em Proceedings of the {I}nternational {C}onference on {A}lgebra,
  {P}art 1 ({N}ovosibirsk, 1989)}, volume 131 of {\em Contemp. Math.}, pages
  13--28, Providence, RI, 1992. Amer. Math. Soc.

\bibitem[BF92]{BesFei92}
M.~Bestvina and M.~Feighn.
\newblock A combination theorem for negatively curved groups.
\newblock {\em J. Differential Geom.}, 35(1):85--101, 1992.

\bibitem[BF96]{BesFei96}
Mladen Bestvina and Mark Feighn.
\newblock Addendum and correction to: ``{A} combination theorem for negatively
  curved groups'' [{J}. {D}ifferential {G}eom.\ {\bf 35} (1992), no.\ 1,
  85--101; {MR}1152226 (93d:53053)].
\newblock {\em J. Differential Geom.}, 43(4):783--788, 1996.

\bibitem[BFH97a]{BesFeiHan97a}
M.~Bestvina, M.~Feighn, and M.~Handel.
\newblock Erratum to: ``{L}aminations, trees, and irreducible automorphisms of
  free groups'' [{G}eom.\ {F}unct.\ {A}nal.\ {\bf 7} (1997), no.\ 2, 215--244;
  {MR}1445386 (98c:20045)].
\newblock {\em Geom. Funct. Anal.}, 7(6):1143, 1997.

\bibitem[BFH97b]{BesFeiHan97}
M.~Bestvina, M.~Feighn, and M.~Handel.
\newblock Laminations, trees, and irreducible automorphisms of free groups.
\newblock {\em Geom. Funct. Anal.}, 7(2):215--244, 1997.

\bibitem[BH99]{BriHae99}
Martin~R. Bridson and Andr{\'e} Haefliger.
\newblock {\em Metric spaces of non-positive curvature}, volume 319 of {\em
  Grundlehren der Mathematischen Wissenschaften [Fundamental Principles of
  Mathematical Sciences]}.
\newblock Springer-Verlag, Berlin, 1999.

\bibitem[Bri00]{Bri00}
P.~Brinkmann.
\newblock Hyperbolic automorphisms of free groups.
\newblock {\em Geom. Funct. Anal.}, 10(5):1071--1089, 2000.

\bibitem[Bur02]{Bur02}
William Burnside.
\newblock On an unsettled question in the theory of discontinuous groups.
\newblock {\em Quart.J.Math.}, 33:230--238, 1902.

\bibitem[CDP90]{CooDelPap90}
M.~Coornaert, T.~Delzant, and A.~Papadopoulos.
\newblock {\em G\'eom\'etrie et th\'eorie des groupes}, volume 1441 of {\em
  Lecture Notes in Mathematics}.
\newblock Springer-Verlag, Berlin, 1990.
\newblock Les groupes hyperboliques de Gromov. [Gromov hyperbolic groups], With
  an English summary.

\bibitem[Che05]{Che05}
E.~A. Cherepanov.
\newblock Free semigroup in the group of automorphisms of the free {B}urnside
  group.
\newblock {\em Comm. Algebra}, 33(2):539--547, 2005.

\bibitem[Cou09]{Cou09}
R{\'e}mi Coulon.
\newblock Asphericity and small cancellation theory for rotation family of
  groups. (arxiv:0907.4577).
\newblock preprint, 2009.

\bibitem[Cou10]{Cou10}
R{\'e}mi Coulon.
\newblock {\em Automorphismes ext\'erieurs du groupe de Burnside libre}.
\newblock PhD thesis, Universit\'e de Strasbourg, june 2010.

\bibitem[Del96]{Del96}
Thomas Delzant.
\newblock Sous-groupes distingu\'es et quotients des groupes hyperboliques.
\newblock {\em Duke Math. J.}, 83(3):661--682, 1996.

\bibitem[DG08]{DelGro08}
Thomas Delzant and Misha Gromov.
\newblock Courbure m{\'e}soscopique et th{\'e}orie de la toute petite
  simplification.
\newblock {\em J. Topol.}, 1(4):804--836, 2008.

\bibitem[GdlH90]{GhyHar90}
{{\'E}}. Ghys and P.~de~la Harpe.
\newblock {\em Sur les groupes hyperboliques d'apr{\`e}s {M}ikhael {G}romov},
  volume~83 of {\em Progress in Mathematics}.
\newblock Birkh{\"a}user Boston Inc., Boston, MA, 1990.
\newblock Papers from the Swiss Seminar on Hyperbolic Groups held in Bern,
  1988.

\bibitem[GL02]{GriLys02}
R.~I. Grigor{\v{c}}uk and I.~G. Lysenok.
\newblock {\em The Burnside problems}.
\newblock Kluwer Academic Publishers, Dordrecht, 2002.

\bibitem[Gre60a]{Gre60}
Martin Greendlinger.
\newblock Dehn's algorithm for the word problem.
\newblock {\em Comm. Pure Appl. Math.}, 13:67--83, 1960.

\bibitem[Gre60b]{Gre60a}
Martin Greendlinger.
\newblock On {D}ehn's algorithms for the conjugacy and word problems, with
  applications.
\newblock {\em Comm. Pure Appl. Math.}, 13:641--677, 1960.

\bibitem[Gre61]{Gre61}
Martin Greendlinger.
\newblock An analogue of a theorem of {M}agnus.
\newblock {\em Arch. Math.}, 12:94--96, 1961.

\bibitem[Hal57]{Hal57}
Marshall Hall, Jr.
\newblock Solution of the {B}urnside problem of exponent {$6$}.
\newblock {\em Proc. Nat. Acad. Sci. U.S.A.}, 43:751--753, 1957.

\bibitem[Lev08]{Lev08}
Counting growth types of automorphisms of free groups.
\newblock 01 2008.

\bibitem[LvdW33]{LevWae33}
Friedrich Levi and B~van~der Waerden.
\newblock \"uber eine besondere klasse von gruppen.
\newblock {\em Abhandlungen aus dem Mathematischen Seminar der Universit\"at
  Hamburg}, 9(1):154--158, Dec 1933.

\bibitem[Mos92]{Mos92}
Brigitte Moss{{\'e}}.
\newblock Puissances de mots et reconnaissabilit{\'e} des points fixes d'une
  substitution.
\newblock {\em Theoret. Comput. Sci.}, 99(2):327--334, 1992.

\bibitem[NA68a]{NovAdj68}
P.~S. Novikov and S.~I. Adjan.
\newblock Infinite periodic groups. {I}.
\newblock {\em Izv. Akad. Nauk SSSR Ser. Mat.}, 32:212--244, 1968.

\bibitem[NA68b]{NovAdj68a}
P.~S. Novikov and S.~I. Adjan.
\newblock Infinite periodic groups. {II}.
\newblock {\em Izv. Akad. Nauk SSSR Ser. Mat.}, 32:251--524, 1968.

\bibitem[NA68c]{NovAdj68b}
P.~S. Novikov and S.~I. Adjan.
\newblock Infinite periodic groups. {III}.
\newblock {\em Izv. Akad. Nauk SSSR Ser. Mat.}, 32:709--731, 1968.

\bibitem[Ol{\cprime}82]{Olc82}
A.~Yu. Ol{\cprime}shanski{\u\i}.
\newblock The {N}ovikov-{A}dyan theorem.
\newblock {\em Mat. Sb. (N.S.)}, 118(160)(2):203--235, 287, 1982.

\bibitem[Ol{\cprime}91]{Olc91}
A.~Yu. Ol{\cprime}shanski{\u\i}.
\newblock Periodic quotient groups of hyperbolic groups.
\newblock {\em Mat. Sb.}, 182(4):543--567, 1991.

\bibitem[Ota96]{Ota96}
Jean-Pierre Otal.
\newblock Le th{\'e}or{\`e}me d'hyperbolisation pour les vari{\'e}t{\'e}s
  fibr{\'e}es de dimension 3.
\newblock {\em Ast{\'e}risque}, (235):x+159, 1996.

\bibitem[San40]{San40}
I.~N. Sanov.
\newblock Solution of {B}urnside's problem for exponent 4.
\newblock {\em Leningrad State Univ. Annals [Uchenye Zapiski] Math. Ser.},
  10:166--170, 1940.

\bibitem[Ser77]{Ser77}
Jean-Pierre Serre.
\newblock {\em Arbres, amalgames, {${\rm SL}\sb{2}$}}.
\newblock Soci{\'e}t{\'e} Math{\'e}matique de France, Paris, 1977.
\newblock Avec un sommaire anglais, R{{\'e}}dig{{\'e}} avec la collaboration de
  Hyman Bass, Ast{{\'e}}risque, No. 46.

\bibitem[Tar49]{Tar49}
V.~A. Tartakovski{\u\i}.
\newblock Solution of the word problem for groups with a {$k$}-reduced basis
  for {$k>6$}.
\newblock {\em Izvestiya Akad. Nauk SSSR. Ser. Mat.}, 13:483--494, 1949.

\end{thebibliography}
